\documentclass[a4paper]{amsart}

\usepackage{amsfonts, amsmath, amssymb, amsthm, epsfig}
\usepackage{thmtools} %
\usepackage[british]{babel}
\usepackage[T1]{fontenc}
\usepackage{standalone,subcaption}
\usepackage[pagebackref=true, hidelinks]{hyperref}
\usepackage[alphabetic, nobysame, backrefs]{amsrefs}
\usepackage[framemethod=TikZ]{mdframed}
\usepackage{todonotes}
\usepackage{color}
\usepackage{lipsum}
\usepackage{tikz}
\usepackage{tikz-3dplot}
\tdplotsetmaincoords{20}{110}
\usepackage[capitalize]{cleveref}

\input{custom-commands}

\usetikzlibrary{calc}

\usepackage{array}
\newcolumntype{C}{>{$}c<{$}} %
	\newcolumntype{R}{>{$}r<{$}} %

\usetikzlibrary{decorations.markings}
\usetikzlibrary{arrows}
\usetikzlibrary{arrows.meta, positioning}
\usetikzlibrary{patterns}

\calclayout

\NewDocumentCommand{\Th}{O{\alpha} m}{\operatorname{Th}_{#1}(#2)}
\NewDocumentCommand{\enlargement}{O{\alpha} m}{\operatorname{e}_{#1}(#2)}
\newcommand{\Out}{O}
\newcommand{\In}{I}
\newcommand{\Z}{\mathbb{Z}}
\newcommand{\vcd}{\mathrm{vcd}}
\newcommand{\cd}{\mathrm{cd}}
\newcommand{\RACG}{right-angled Coxeter group\xspace}
\newcommand{\RACGs}{right-angled Coxeter groups\xspace}

\newcommand{\fns}{flag-no-square\xspace}%
\newcommand{\cdim}{\operatorname{Cdim}}

\newcommand{\scx}{M} %
\newcommand{\varscx}{N} %
\newcommand{\thickmap}[1][\scx]{\psi} %
\newcommand{\thicksize}{h}

\author{Giovanni Italiano}
\address{\parbox{\linewidth}{Mathematical Institute, University of Oxford, \\Andrew Wiles Building,
		Woodstock Road, OX2 6GG Oxford, UK}\vspace{1.5pt}}
\email{italiano at maths dot ox dot ac dot uk}

\author{Matteo Migliorini}
\address{Karlsruher Institut für Technologie, Englerstraße 2, 76131 Karlsruhe}
\email{matteo dot migliorini at kit dot edu}

\author{Andrew Ng}
\address{Mathematisches Institut, Universität Bonn, Endenicher Allee 60, 53115 Bonn, Germany}
\email{\href{mailto:clan@math.uni-bonn.de}{clan@math.uni-bonn.de}}

\title{Improved algebraic fibrations of high-dimensional hyperbolic groups}


\begin{document}

\begin{abstract}
	For every $d \geq 3$, we construct infinitely many quasi-isometry classes of hyperbolic groups $G$ of cohomological dimension $d$ that algebraically fibre with finitely presented kernel. All our groups arise as finite-index subgroups of right-angled Coxeter groups. In many cases, the $L^2$-Betti numbers of the groups $G$ provide obstructions to higher finiteness properties of the kernel. Our groups therefore expand the list of subgroups of hyperbolic groups with exotic finiteness properties.

\end{abstract}

\maketitle

\section{Introduction}

A group $G$ is said to \newterm{algebraically fibre} if it admits a surjective map onto $\Z$ with finitely generated kernel. This definition is motivated by the world of $3$-manifolds, where a celebrated theorem of Stallings \cite{Stallings} establishes that, for a compact irreducible $3$-manifold, an algebraic fibration of its fundamental group is always induced by a topological fibration of the manifold. Agol's virtual fibring criteria \cite{Agol2008}, that he used to prove that all finite volume hyperbolic $3$-manifolds virtually fibre, makes crucial use of this. In view of this, the next natural question is whether all odd dimensional, finite volume hyperbolic manifolds virtually fibre.

In higher dimensions, stronger hypotheses are needed to upgrade algebraic fibring to topological fibring as, in general, a topological fibration induces a map on fundamental groups where the kernel is of type $F$. Finding examples of higher dimensional hyperbolic manifolds whose fundamental groups algebraically fibre with intermediate finiteness properties has been a topic of interest of late \cites{BattistaMartelli, IMM48, IM6}. The ultimate goal, however, is to obtain examples where the manifold topologically fibres. In his thesis \cite{Farrell}, Farrell proves a theorem, which holds in dimensions $d\geq 6$, in a similar spirit to Stallings' theorem, that says that an algebraic fibring of a $d$-manifold group is induced by a topological fibring if and only if the kernel is of type $F$ and certain $K$-theoretic obstructions vanish. Following Farrell's result, one ingredient to produce topological fibrings would be to find algebraically fibrings where the finiteness properties of the kernel can be improved.%

Agol's virtual fibring criterion crucially makes use of the notion of a RFRS group. We will not need the definition of a RFRS group, but it suffices to know that all compact special groups in the sense of Haglund and Wise \cite{HaglundWise2008} are RFRS. It is known that arithmetic hyperbolic manifolds of first type have fundamental groups which are cubulated \cite{BergeronWise}, hence virtually compact special \cite{Agol_haken} and therefore virtually RFRS. Wise conjectured that all finite volume hyperbolic manifolds are cubulated \cite[Section 17.b]{Wise_invitation}. In view of this, a lot of research has focussed on the special case of fibring manifolds whose fundamental group is RFRS.

Let $M$ be a closed hyperbolic odd-dimensional manifold and $G=\pi_1(M)$. To use Farrell's Theorem, one crucially needs to find a fibration $G \twoheadrightarrow \Z$ where the kernel is type $F$. One promising approach to finding a suitable fibration of a RFRS group is the following strategy, which is discussed in more detail in Kielak's companion notes to his ICM talk this year \cite{KielakICM}.
\begin{enumerate}
	\item Finding an algebraic fibration $G \twoheadrightarrow \Z$;
	\item Proving that the kernel is finitely presented;
	\item Proving that the kernel is of type $FP_n(\Z)$ for every $n$.
	\item Deducing that the kernel has actually property $F_\infty$, as a finitely presented group with property $FP_n(\Z)$ has also property $F_n$.
\end{enumerate}

In \cite{K20} Kielak showed the equivalence between algebraic fibring of an infinite RFRS group and the homological condition of vanishing of its first $\ell^2$-Betti number $b_1^{(2)}(G)$. This opened up the possibility of bringing homological algebra to bear on this problem, and indeed Fisher, building upon Kielak's Theorem, proved that, for any field $\mathbb{K}$, a RFRS group virtually algebraically fibres with kernel of type $FP_n(\mathbb{K})$ if and only if the first $n$ $\ell^2$-Betti numbers with coefficients in $\mathbb{K}$ vanish \cite{F24}.

Kielak's theorem provides techniques to prove $(1)$, and Fisher's theorem allows us to venture quite close to $(3)$. However, there is a lack of general techniques to approach $(2)$, partly because it appears to be quite resistant to purely homological methods. It is therefore interesting to understand more broadly fibrations of hyperbolic RFRS groups with the goal of bringing these techniques back to the setting of hyperbolic manifolds. \footnote{When we speak about hyperbolic manifolds, we mean manifolds that admit a metric of constant sectional curvature $-1$, while by hyperbolic group we mean word hyperbolic group. It should be clear from context which definition is meant and we ask the reader's understanding in using the same word in two senses.}

Recently Lafont, Minemyer, Sorcar, Stover, and Wells \cite{LMSSW} constructed explicit examples of hyperbolic right-angled Coxeter groups that virtually algebraically fibre, in each virtual cohomological dimension $d\geq 2$. Their result relies on a recursive construction, which is an adaptation of a thickening due to Osajda \cite{Osajda} to access high cohomological dimensions, and on the combinatorial game of Jankiewicz, Norin and Wise \cite{JNW} to ensure the virtual algebraic fibring. These groups are not fundamental groups of manifolds, but they are the first example in which algebraic fibring is shown to happen in hyperbolic groups of arbitrarily high dimension. %
\medskip

In this paper we strengthen the result of \cite{LMSSW} by proving the following.
\begin{maintheorem}\label{main:fp-fibring}
	For every $d\geq 3$, there exist infinitely many isomorphism classes of hyperbolic groups of cohomological dimension $d$ that algebraically fibre with finitely presented kernel.
\end{maintheorem}

All our examples arise as finite index subgroups of right-angled Coxeter groups. Our method is to combine the strategy of \cite{LMSSW} with a novel coning construction, which we call \emph{collar-coning}, that allows us to modify simplicial complexes to be simply-connected while maintaining the hyperbolicity of the \RACG that they define.
The main strength of Theorem \ref{main:fp-fibring} is promoting the algebraic fibring by proving $(2)$ of the strategy outlined above for an infinite family of groups. To the best of our knowledge, before this paper there were only some sparse examples of hyperbolic groups of cohomological dimension $d>3$ for which $F_2$-algebraic fibring has been already proven, see \cites{IMM5, IM6, IMM48}. %
Our examples also contribute to the list of hyperbolic groups with finitely presented subgroups that have exotic finiteness properties \cites{BradyF2notF3, LodhaF2notF3, KropF2notF3, IM6, LMP8, LlosaPy}. Our construction has the advantage of providing an infinite family, with precise control over the dimension, as well as the nice combinatorial structure coming from the ambient \RACG.

Moreover, using the same technique as in recent work of Cashen, Dani, Schreve, and Stark \cite{CDSS}, we can distinguish infinitely many quasi-isometry classes of such groups in every fixed dimension. Our theorem is therefore upgraded to the following.

\begin{maintheorem}\label{main:fp-fibring-qi}
	For every $d\geq 3$, there exist infinitely many quasi-isometry classes of hyperbolic groups of cohomological dimension $d$ that algebraically fibre with finitely presented kernel.
\end{maintheorem}

The finitely presentability of the kernel implies the vanishing of the first two $\ell^2$-Betti numbers of the group \cite{F24}. Since the Euler characteristic of a group is equal to the alternating sum of its $\ell^2$-Betti numbers, its non-vanishing allows us to obstruct some finiteness properties. In particular, we get the following.

\begin{maintheorem}\label{main:fp-fibring-notFn}
	For every $n\geq 3$, there exist infinitely many quasi-isometry classes of hyperbolic groups of cohomological dimension $n$ that algebraically fibre with kernel of type $F_2$ but not $F_n$. Moreover, if $n$ is even, then the same result with $F_{n-1}$ in place of $F_{n}$ holds.
\end{maintheorem}

We believe our result is new even in dimension 3, although of course the most interesting case of hyperbolic groups of cohomological dimension 3, namely hyperbolic 3-manifolds, is dealt with by Agol's resolution of the virtual fibring conjecture \cite{Agol_haken}.
\begin{maincorollary}\label{main:fp-fibring-dim3}
	There exist infinitely many quasi-isometry classes of hyperbolic groups of cohomological dimension $3$ that algebraically fibre with kernel of type $F_2$ but not $F_3$.
\end{maincorollary}
In the special case $n=4$ of \cref{main:fp-fibring-notFn}, we obtain the following corollary, where the behaviour significantly differs from the case of hyperbolic manifold case.
\begin{maincorollary}\label{main:fp-fibring-dim4}
	There exist infinitely many quasi-isometry classes of hyperbolic groups of cohomological dimension $4$ that algebraically fibre with kernel of type $F_2$ but not $F_3$.
\end{maincorollary}
Previous work of Kropholler \cite{KropF2notF3} produced examples of cohomological dimension 3, but in the closing remarks of the paper Kropholler discusses the difficulties of using his methods to obtain the analogous result in higher cohomological dimensions.

To the best of our knowledge, these are the first examples of hyperbolic groups of cohomological dimension 4 that fibre with kernel of type $F_2$, and even the first examples of hyperbolic groups of cohomological dimension 4 with vanishing second $L^2$-Betti number. This contrasts with the case of hyperbolic 4-manifolds, where the proof of the Singer conjecture for such manifolds \cite{Dodziuk} implies that the second $L^2$-Betti number is non-zero (and in fact the only non-zero $L^2$-Betti number). Moreover, the Singer conjecture predicts that the same is true for all fundamental groups of closed 4-manifolds that admit a metric of strictly negative sectional curvature.

\subsection*{Structure of the paper}
The paper is structured as follows.

\begin{itemize}
	\item In \cref{sec:preliminaries} we recall some preliminary definitions and results.
	\item In \cref{sec:thickening} we give a brief exposition of the JNW game to fix notation, and prove a slight generalisation of the techniques of \cite{LMSSW}.
	\item In \cref{sec:coning} we describe our collar-coning procedure that, together with the previous thickening construction, is the main ingredient for constructing our family of groups.
	\item In \cref{sec:theorems} we prove the results stated in the introduction.%

	\item In \cref{sec:next} we discuss possible further directions.
\end{itemize}

\subsection*{Acknowledgments}
The first author gratefully acknowledges support from the Royal Society through the Newton International Fellowship (award number: NIF\textbackslash R1\textbackslash 231857).
The second author gratefully acknowledges funding by the DFG 281869850 (RTG 2229).
The first and second authors are members of the INdAM GNSAGA research group. The third author gratefully acknowledges funding from the European Union (ERC, SATURN, 101076148) and the Deutsche Forschungsgemeinschaft (EXC-2047/1 - 390685813).

The authors would like to thank the Isaac Newton Institute for Mathematical Sciences, Cambridge, for support and hospitality during the programme Operators, Graphs, and Groups, where work on this paper was undertaken. This work was partially supported by EPSRC grant no EP/R014604/1. The authors also thank the organisers of the conference `Higher dimensional hyperbolic geometry' in Ventotene, Italy, where work on this paper was undertaken. The authors thank Kevin Schreve for pointing out relevant references.

This paper forms part of the third author's PhD thesis.

\textit{For the purpose of Open Access, the authors have applied a CC BY public copyright licence to any Author Accepted Manuscript (AAM) version arising from this submission.}

\section{Preliminaries}\label{sec:preliminaries}

\subsection{Right-angled Coxeter groups}
\begin{definition}
	Given a simplicial graph $\Gamma$ with vertex set $V$ and edge set $E$, the associated right-angled Coxeter group $W_\Gamma$ is the group with the following presentation:
	\[ W_\Gamma := \langle v\in V \mid v^2=1 \ \forall v\in V,\ [v,w]=1 \ \forall (v,w)\in E \rangle \]
\end{definition}
There is a cube complex ${X_\Gamma}$ associated to a \RACG $W_\Gamma$, called the \newterm{Davis complex}. It is a $CAT(0)$ cube complex, whose $1$-skeleton corresponds to the Cayley graph of the group $W_\Gamma$ (with respect with the set of generators given by the vertices of $\Gamma$, and after having identified every bigon to a single edge), and such that the link of every vertex is isomorphic to the flag completion $\flag{\Gamma}$ of the graph $\Gamma$. A more complete description can be found in \cite{Davisbook}.

Using the identification with the Cayley graph, every edge of ${X_\Gamma}$ is labelled by one of the generators of the group. For this reason, there is a canonical isomorphism between the link of a vertex of ${X_\Gamma}$ and $\flag{\Gamma}$, which respect the labelling.

It is a well known fact that hyperbolicity of a \RACG can be verified from a combinatorial condition on the defining graph.
\begin{proposition} \cite{MoussongThesis88}
	A \RACG $W_\Gamma$ is word-hyperbolic if and only if its defining graph $\Gamma$ has no induced squares.
\end{proposition}

\subsection{The Jankiewicz-Norin-Wise combinatorial game}
In \cite{JNW}, Jankiewicz, Norin, and Wise introduced a combinatorial game to construct virtual maps from a \RACG to $\Z$, in a way that is conducive to studying finiteness properties of the kernel. The main underlying tool for this analysis is Bestvina-Brady Morse theory \cite{BestvinaBrady}.
We use a slight modification of it, and we briefly summarize the procedure here.

\begin{definition}
	Let $\Gamma$ be a simplicial graph.
	\begin{enumerate}
		\item A \newterm{state} on a simplicial graph $\Gamma$ is the partition of its vertices into two sets, called $I$ and $O$, standing for ``IN'' and ``OUT''.
		\item A \newterm{(coloured) system of moves} on $\Gamma$ is a partition of the set of vertex in a finite number of colours $V=\bigsqcup_{i=1}^{c}V_i$. The sets $V_1, \dots, V_c$ are often referred to as the \textit{colours} $\{1, \dots, c\}$, or the \textit{moves}. In this paper we will only consider system of moves which are \newterm{sparse}, i.e. for which adjacent vertices belong to different moves.
	\end{enumerate}
\end{definition}

Moves act on states in the following way. Given the move $m_i$ and a state $S$, the state $m_i(S)$ is the state where every vertex in $V\setminus V_i$ maintains the same label $I/O$ as the one on $S$, while the vertices in $V_i$ switch to the opposite label to the one in $S$.
\smallskip

The \textit{(coloured) JNW combinatorial game} starts with the assignment of a (coloured and sparse) system of moves, and the choice of a starting state $S_0$ on $\Gamma$. Let $c$ be the number of moves in the system, and let $\alpha \colon W_\Gamma \to (\Z/2\Z)^c$ the homomorphism that maps every generator $v$ to the element of the canonical basis $e_{m(v)}\in(\Z/2\Z)^c$ corresponding to the move $m(v)$ assigned to $v$. Let $G'=\ker \alpha$, and $X'=G'\backslash{X_\Gamma}$. Note that $G'$ is a finite-index subgroup of $W_\Gamma$ (of index $2^c$), $X'$ is a non-positively curved cube complex with $2^c$ vertices, and such that $\pi_1(X')=G'$. Moreover, $X'$ inherits the labelling of the edges from $X_\Gamma$, and so in particular also the natural isomorphisms of each vertex links with the flag completion $\flag{\Gamma}$ of $\Gamma$. %

\begin{enumerate}
	\item Fix $x_0\in X'$ a base vertex. On the link of $x_0$ we assign the starting state $S_0$. Since every edge of $X'$ is labelled by a $v\in V(\Gamma)$, it is also labelled by the move associated to $v$. We can now propagate the state $s_0$ to the other vertices. By choosing a path connecting $x_0$ to another vertex $x$, and applying all the moves appearing in the edges of the paths to $s_0$, we obtain a new state $s$. It is possible to verify that $s$ does not depend on the choice of the path, and it is therefore a well-posed definition for a state on $x$.
	\item Each edge of $X'$ inherits opposed status I and O at the two vertices, this allows us to define a direction on it (going from OUT to IN). The fact that the system of moves is sparse ensures that on every square parallel edges get the same direction.
	\item This coherence property on parallel edges implies that each cube has exactly one vertex where all edges have arrows that are going out.
	      We can identify each $d$-cube with the standard cube $[0,1]^d$, by setting this vertex to be $(0,\dots,0)$.
	\item Following the directions on the edges, we define a \newterm{diagonal map} onto $S^1$ on each cube of $X'$, using the composition $[0,1]^d\to \mathbb{R} \to \mathbb{R}/\Z=S^1$ given by $(x_1,\dots,x_d)\mapsto \sum x_i$ and the projection to the quotient.
	\item This map is a Bestvina-Brady Morse function, and its ascending (resp. descending) links correspond to the subcomplex of $\tilde{\Gamma}$ spanned by the vertices labelled with $O$ (resp. $I$). This induces a map $ G' = \pi_1(X') \to \pi_1(S^1) = \Z $, and we can infer properties of its kernel by using the following theorem from \cite{BestvinaBrady}.
\end{enumerate}

\begin{theorem}[\cite{BestvinaBrady}]
	Let $f \colon X \to S^1$ Morse function and let $H = \ker{f_*}$.
	Suppose that each ascending link and each descenting link is $k$-connected. Then $H$ is of type $F_{k+1}$.

\end{theorem}

In particular, we will be looking for an initial state and a system of moves that produce an orbit in which all ascending and descending links are $k$-connected. Following the terminology from \cite{JNW}, we will call such an orbit \newterm{$k$-legal}.

\subsection{Thickening and high virtual cohomological dimension}
The \newterm{cohomological dimension} of a group $G$ is defined as the supremum of the $n$ such that $H^n(G,M)\neq 0$ for some $\Z G$-module $M$. If $G$ has torsion, its cohomological dimension is always $\cd G = \infty$. For virtually torsion-free groups, the usual notion of dimension that is used is \newterm{virtual cohomological dimension}, denoted $\vcd G$, which is the cohomological dimension of any torsion-free finite index subgroup of $G$ (this is indeed independent of the choice of such subgroup).

The virtual cohomological dimension of a \RACG can be explicitly computed from the defining graph $\Gamma$.
\begin{theorem}\cite[Corollary 8.5.5] {Davisbook}
	Let $\Gamma$ be a simplicial graph, and let $L$ be its flag completion. Then
	\[ \vcd W_\Gamma = \max \{n \colon \bar{H}^{n-1}(\flag{\Gamma}\setminus \sigma) \neq 0 \text{ for some simplex $\sigma$} \}. \]
\end{theorem}

There are examples of hyperbolic \RACGs of arbitrarily high virtual cohomological dimension. This fact was somewhat surprising at the time of discovery, as it does not hold for Coxeter polytopes, nor for Coxeter groups which are Poincaré duality groups (see, for example, \cite{Vinberg85}). It was indeed conjectured that there should be an universal bound on the virtual cohomological dimension of any word hyperbolic Coxeter group. The conjecture first appeared in print in Moussong's thesis \cite{MoussongThesis88} and was disproved by Januszkiewicz--\'Swi\k atkowski \cite{JS03} and later independently by Osajda \cite{Osajda}, who constructed families of hyperbolic Coxeter groups of every possible virtual cohomological dimension. Both of their constructions are purely combinatorial.

\smallskip

We recall Osajda's construction, as our argument relies on a modification of it.

\begin{definition}\label{def:osajda}
	Given a cube complex $X$, the thickening $\Th[1]X$ is the flag simplicial complex with the same set of vertices as $X$, and where vertices $v_0, \dots, v_n $ span a simplex if and only if they are contained in a common cube in $X$.
\end{definition}
Osajda's construction proceeds recursively in the following way:
\begin{enumerate}
	\item We start with a finite flag-no-square simplicial complex $\Delta_n$ with $H^{n-1}(\Delta_n)\neq 0$, take $\Gamma_n$ to be its $1$-skeleton, $W_n$ to be the \RACG associated to $\Gamma_n$, and $X_n$ its Davis' complex.
	\item We select a torsion-free subgroup $H<W_n$ of sufficiently large finite index. Then the quotient $Y_n= X_n / H$ is a negatively curved cube complex, and we can prove that $H^{n}(Y_n)\neq 0$.
	\item We define $\Delta_{n+1}=\Th[1]{Y_n}$ to be the thickening. Then $\Delta_{n+1}$ is a finite flag-no-square simplicial complex with $H^{n}(\Delta_{n+1})\neq 0$, and we can proceed inductively.
\end{enumerate}

\subsection{Conformal dimension}
In order to prove the quasi-isometry part of \cref{main:fp-fibring-qi} we will distinguish our examples by the conformal dimensions of their visual boundaries, following \cite{CDSS}. Here we introduce the necessary background material on conformal dimension. For a more detailed treatment the reader is referred to the monograph \cite{MT}.
\begin{definition}
	Let \((M, d)\) and \( (M', d')\) be metric spaces. A homeomorphism \(f \colon M \to M'\) is a \emph{quasisymmetry} if there exists a homeomorphism \(\varphi \colon [0, \infty) \to [0, \infty)\) so that for distinct \(x, y, z \in M\):
	\[
		\frac{d'(f(x), f(y))}{d'(f(x), f(z))} \leq \varphi\left(\frac{d(x, y)}{d(x, z)}\right).
	\]
	The spaces \((M, d)\) and \((M', d')\) are \emph{quasisymmetric} if there exists a quasisymmetry \(f \colon M \to M'\).
\end{definition}

\begin{definition}
	Let \((M, d)\) be a metric space. The \emph{conformal dimension} of \(M\), denoted $\cdim(M)$, is the infimal Hausdorff dimension of any metric space \((M', d')\) quasisymmetric to \((M, d)\).
\end{definition}

Conformal dimension is monotone in the sense that $A \subset B$ implies $\cdim(A) \leq \cdim(B)$ \cite[Section 2.2]{MT}. This definition was introduced by Pansu \cite{P89}, who realised its usefulness in distinguishing metric spaces up to quasi-isometry. We discuss this below.

\begin{definition}
	Let $M$ be a proper geodesic metric space. The visual boundary of $M$ is denoted $\partial M$ and consists of all equivalence classes of geodesic rays, where two rays $\gamma, \gamma' \colon [0, \infty) \to M$ are equivalent if there exists a constant $C$ so that $d(\gamma(t), \gamma' (t)) \leq C$ for all $t \in [0, \infty)$.
\end{definition}

It turns out that for $\delta$-hyperbolic proper geodesic metric spaces $M$, this induces a metric on the boundary $\partial M$. Furthermore, the conformal dimension of the Gromov boundary with this metric is a quasi-isometry invariant for $\delta$-hyperbolic spaces with uniformly perfect boundaries \cite[Corollary 3.2.14]{MT}. In particular, this includes all word hyperbolic groups, for which the conformal dimension is finite \cite{P97}. In general, conformal dimension takes values in $\{0\} \cup [1, \infty]$ and all possible values are realised \cite{K}. As Pansu showed, this invariant can distinguish spaces with homeomorphic boundaries.

\section{Thickening and JNW game}\label{sec:thickening}

In this section we recall the main constructions and results of \cite{LMSSW} that are needed for this paper.

In \cite{LMSSW} the authors introduce the $\alpha$-thickening to make the Jankiewicz--Norin--Wise combinatorial game work. For our purposes, it is convenient to split this construction as a composition of Osajda's thickening (see \cref{def:osajda}) and an operation defined on simplicial complexes that we call $\alpha$-enlargement.

\begin{definition}
	Let $\scx$ be a simplicial complex, and let $\alpha \colon Y \to \skel0\scx $ be a surjective map. The $\alpha$-enlargement is a simplicial complex $\enlargement \scx$, equipped with a simplicial map $\bar\alpha \colon \enlargement[\alpha]\scx \to \scx$, defined as follows.
	\begin{itemize}
		\item The vertex set of $\enlargement[\alpha]\scx$ is $Y$.
		\item Vertices $v_0, \dots, v_n $ of $\enlargement[\alpha] \scx$ span a simplex if and only if $\set{ \alpha(v_0), \dots, \alpha(v_n) }$ is the vertex set of a simplex in $\scx$.
		\item The map $\bar\alpha$ is the unique simplicial map that extends $\alpha$.
	\end{itemize}
\end{definition}

\begin{remark}
	Note that a surjective simplicial map $f \colon \varscx \to \scx$ arises as an $\alpha$-enlargement if and only if preimages of simplices in $\scx$ are simplices in $\varscx$. In this case, it is easy to see that $\varscx = \enlargement[\alpha]\scx$, where $\alpha = \restrict f{\skel0\varscx}$.
\end{remark}

The $\alpha$-thickening $\Th X$ of the cube complex $X$, described in \cite{LMSSW}, coincides with $\enlargement{\Th[1]X}$. If $\alpha$ is the identity map, then $\Th X$ coincides with the Osajda thickening $\Th[1] X$.

\begin{lemma}
	The map $\bar \alpha \colon \enlargement[\alpha]\scx \to \scx$ is a homotopy equivalence.
\end{lemma}
This property is standard and was also used in \cite{LMSSW}; we recall the proof for the reader's convenience.

\begin{proof}
	Choose any embedding $\iota \colon \scx \to \enlargement[\alpha] \scx $ so that $\bar\alpha \circ \iota $ is the identity. We claim $\enlargement[\alpha]\scx$ collapses to the image of $\iota$. To see this, note that if $v$ is a vertex not in the image of $\iota$, then its link is a cone on $\iota(\alpha(v))$, so the removal of the open star of $v$ can be achieved by a sequence of simplicial collapses. Iterating this procedure yields the claim, which in turn proves the lemma.
\end{proof}

We now adapt the notion of \emph{disconnecting cubes} and \emph{isolated vertices}, introduced in \cite{LMSSW}, from the cube complex to its thickening.

\begin{definition}
	Let $\scx$ be a simplicial complex. A simplex $\sigma \in \scx$ is $k$-disconnecting if $\scx \setminus \sigma $ is not $k$-connected. We say that $\scx$ has no $k$-disconnecting simplices if $\scx$ is $k$-connected and no simplex of $\scx$ (of any dimension) is $k$-disconnecting.
\end{definition}

Via a recursive construction, \cite{LMSSW} prove the following.

\begin{theorem}
	For every $n \geq 2$ there exist a cube complex $X_n$ such that every thickening $ \Th {X_n}$:
	\begin{itemize}
		\item is \fns;
		\item has no $0$-disconnecting simplices;
		\item defines a Coxeter group with virtual cohomological dimension $n+1$.
	\end{itemize}
\end{theorem}

For the thickening, the authors consider in particular the map $\thickmap \colon Y \to \skel0\scx$, where
\[
	Y = \set{ (v,w) \in \skel0 \scx \times \skel0 \scx : d(v,w) \geq 2 },
\]
and $\thickmap(v,w)=v$ is the projection onto the first factor. The map $\thickmap$ is surjective as long as $\scx$ is not a cone, that is, no vertex of $\scx$ is adjacent to every other vertex of $\scx$.

The following was implicitly proved in \cite{LMSSW} for $k=0$, and its generalisation to every natural $k$ has a similar proof.

\begin{proposition}\label{prop:alpha-thickening-k-fibres}
	Let $\scx$ be a simplicial complex, and let $k \in \NN$. Assume that $\scx$ is not a cone, and has no $k$-disconnecting simplices. Then the Coxeter group associated to the enlargement $\enlargement[\thickmap] \scx$ virtually fibres with kernel of type $\finitetype[k+1]$.
\end{proposition}

\begin{proof}
	Consider the colouring of the vertices of $\enlargement[\thickmap] \scx$ obtained by partitioning them into pairs of the form $\set{(v,w), (w,v)}$. A \emph{balanced state} is a state where vertices of the same colour have opposite states. It is straightforward to see that the orbit consists of all the balanced states. Choose any balanced state as the initial state.

	We want to prove that the orbit is $k$-legal, that is, all the ascending and descending links are $k$-connected. We show it for the ascending links, as for the descending links it is analogous.

	Fix a state of the orbit, and let $Z$ be the subcomplex of $\scx$ spanned by vertices $v$ such that $(v,w)$ has status $\Out$ for some $w$. Now note that any two vertices $v,w$ of $\scx \setminus Z$ are necessarily connected by an edge. If that were not the case, then $(v,w), (w,v)$ would both have status $\In$, which contradicts the assumption that we started with a balanced state.

	Hence, $Z$ is of the form $\scx \setminus \sigma$ for some simplex $\sigma$, and is by hypothesis $k$-connected. Since the ascending link is a thickening of $Z$ and is therefore homotopy equivalent to it, we conclude by the main result of \cite{BestvinaBrady}.
\end{proof}

We can generalise this result to any enlargement that is large enough. To formalize this, for a given simplicial complex $\scx$ and for every $\thicksize \in \NN$ define $\pi_\thicksize \colon \skel0\scx \times \range{\thicksize} \to \skel0 \scx$ to be the projection.

\begin{theorem} \label{verythick}
	Let $\scx$ be a simplicial complex, and let $k \in \NN $. Assume that $\scx$ is not a cone, and has no $k$-disconnecting simplices. Then for $\thicksize \geq \cardinality{\skel0\scx}$ the Coxeter group associated to the enlargement $\enlargement[\pi_\thicksize] M$ virtually fibres with kernel of type $\finitetype[k+1]$.
\end{theorem}

\begin{proof}
	Let $\thickmap \colon Y \to \skel0\scx$ be the enlargement map used in \cref{prop:alpha-thickening-k-fibres}.
	Since $\thicksize \geq \cardinality{\skel0\scx}$, we can find an injective map $\iota \colon Y \to \skel0\scx \times \range \thicksize$ such that $\pi_\thicksize \circ \iota = \thickmap$. This uniquely extends to an embedding $\bar \iota \colon \enlargement[\thickmap]\scx \to \enlargement[\pi_\thicksize]\scx$, satisfying $\overline{\pi_\thicksize} \circ \bar \iota = \bar\thickmap$.

	We define the following colouring on the vertices of $\enlargement[\pi_\thicksize] \scx$. On the image of $\iota$, we use the colouring induced by the colouring on $\enlargement[\thickmap]\scx$ described in the proof of \Cref{prop:alpha-thickening-k-fibres}. All other vertices are coloured with distinct colours, i.e.~a vertex outside the image of $\iota$ does not have the same colour as any other vertex. We choose again any balanced state as initial state.

	The same argument as in \Cref{prop:alpha-thickening-k-fibres} allows us to conclude that all ascending and descending links are $k$-connected.
\end{proof}

\section{The coning procedure}\label{sec:coning}

In this section we present an algorithm for producing the simplicial complexes needed for the proofs of \cref{main:fp-fibring} and \cref{main:fp-fibring-qi}. We will refer to the fundamental groups $\pi_1(M \setminus \sigma)$ as $\sigma$ ranges over all possible simplices of $M$, so it will be inconvenient to consider based loops. Hence, for the purposes of this section, when we discuss generating a group, we will often consider collections of \emph{free homotopy classes} of loops that \emph{normally} generate $\pi_1(M \setminus \sigma)$. By abuse of notation, we use the same symbol for a loop and the conjugacy class of the element in $\pi_1(M \setminus \sigma)$ that it represents.

Given a simplicial complex $N$, $N \setminus \sigma$ refers to the induced subcomplex on the set of vertices $N^{(0)} \setminus \sigma^{(0)} $. Recall that a subcomplex $M \subset N$ is full if, whenever $n+1$ points span an $n$-simplex in $N$, they span an $n$-simplex in $M$. The goal of this section is to prove the following:

\begin{theorem} \label{coning-procedure}
	Let $M$ be a \fns simplicial complex of homological dimension $d \geq 2$ with no $0$-disconnecting simplices. Then there is a \fns simplicial complex $N$, also of homological dimension $d$, which contains $M$ as a full subcomplex and has no $1$-disconnecting simplices.
\end{theorem}

The idea is to cone off generators for $\pi_1(M \setminus \sigma)$ as $\sigma$ ranges over all subsimplices of $M$. This must be done with some care to ensure that no squares are created. We use a device which we call a ``collared cone''.

\begin{definition}
	For $n \geq 5$, define the \emph{collared cone} on an $n$-cycle to be the \fns triangulated disk $C$ constructed as following:
	\begin{enumerate}
		\item Start from an annulus, and subdivide it into $n$ squares, so that its boundary components are $n$-cycles. This is the \emph{collar} of the collared cone.
		\item Add a diagonal to each square such that no two diagonals have the same endpoint. %

		\item Finally cone one of the boundary components of the collar to a point, see \cref{fig:collared-cone}.
	\end{enumerate}
	Let $M$ be a simplicial complex, and let $\gamma$ be an $n$-cycle in $M$. \emph{`Collar-coning the $n$-cycle $\gamma$'} refers to the procedure of attaching the boundary of the collared cone $C$ to $\gamma_n$.
\end{definition}

\begin{remark}\label{rmk:Lobell}
	The collared cone on the pentagon is the complement of the star of a vertex of an icosahedron, see \cref{fig:truncated-icosahedron}. More generally, the collared cone on the $n$-gon is the complement of the star of a vertex in the dual of the Löbell polyhedron $L(n)$ \cite{Lobell}. The crucial property that will allow our construction to work is the fact that the Löbell polyhedra can be represented as right-angled polyhedra in $\mathbb{H}^3$.
\end{remark}

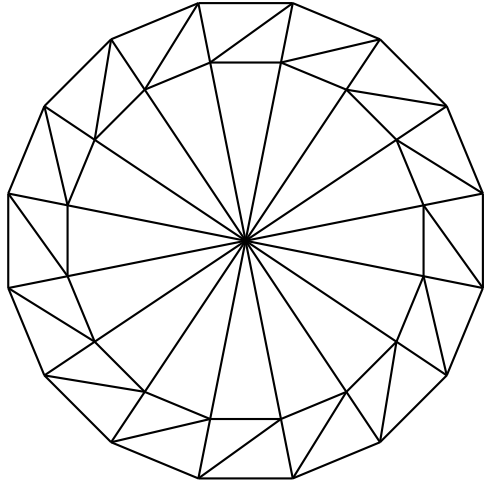
\begin{figure}
	\begin{tikzpicture}[scale=0.8, line width=0.20mm]
	\usetikzlibrary{calc}

	\def\n{16}

	\def\outr{4}

	\def\inr{3}

	\foreach \ii in {1,...,\n} {
			\def\ang{360/\n};
			\def\curang{360/\n*\ii+\ang/2+90};
			\def\stagang{\curang};
			\draw[thick] (\stagang:\outr) -- (\stagang+\ang:\outr);
			\draw[thick] (\stagang:\outr) -- (\curang:\inr);
			\draw[thick] (\curang:\inr) -- (\curang+\ang:\inr);
			\draw[thick] (\stagang:\outr) -- (\curang+\ang:\inr);
			\draw[thick] (\curang:\inr) -- (0,0);
		};
\end{tikzpicture}
	\caption{A collared cone.}
	\label{fig:collared-cone}
\end{figure}

\begin{figure}
	\centering
	\begin{minipage}{0.48\linewidth}
		\centering
		\begin{tikzpicture}[tdplot_main_coords, scale=2.0, line join=round]

\def\phi{1.618033988749895}

\coordinate (V3)  at ( 0,  1, -\phi);
\coordinate (V4)  at ( 0, -1, -\phi);
\coordinate (V5)  at ( 1,  \phi,  0);
\coordinate (V6)  at (-1,  \phi,  0);
\coordinate (V7)  at ( 1, -\phi,  0);
\coordinate (V9)  at ( \phi,  0,  1);
\coordinate (V10) at (-\phi,  0,  1);
\coordinate (V11) at ( \phi,  0, -1);

\draw[very thick, color=blue] (V5) -- (V11) -- (V7); 
\filldraw[
    shade,
    inner color=red!5,
    outer color=red!15,
    draw=red,
    very thick
] (V3) -- (V5) -- (V9) -- (V7) -- (V4) -- cycle;
\draw[very thick] (V5) -- (V10)  (V7) -- (V10)  (V9) -- (V10);
\draw (V3) -- (V10)  (V4) -- (V10);

  \fill (V3) circle (0.02);
  \fill (V4) circle (0.02);
  \fill (V5) circle (0.02);
  \fill (V7) circle (0.02);
  \fill (V9) circle (0.02);
  \fill (V10) circle (0.02);
  \fill (V11) circle (0.02);

\fill[color=white] (V6) circle (0.02);

\end{tikzpicture}
	\end{minipage}\hfill
	\begin{minipage}{0.48\linewidth}
		\centering
		\begin{tikzpicture}[tdplot_main_coords, scale=2.0, line join=round]

\def\phi{1.618033988749895}

\coordinate (V1)  at ( 0,  1,  \phi);
\coordinate (V2)  at ( 0, -1,  \phi);
\coordinate (V3)  at ( 0,  1, -\phi);
\coordinate (V4)  at ( 0, -1, -\phi);
\coordinate (V5)  at ( 1,  \phi,  0);
\coordinate (V6)  at (-1,  \phi,  0);
\coordinate (V7)  at ( 1, -\phi,  0);
\coordinate (V8)  at (-1, -\phi,  0);
\coordinate (V9)  at ( \phi,  0,  1);
\coordinate (V10) at (-\phi,  0,  1);
\coordinate (V11) at ( \phi,  0, -1);
\coordinate (V12) at (-\phi,  0, -1);

\draw[very thick, color=blue] (V5) -- (V11) -- (V7); 
\filldraw[
    shade,
    inner color=red!5,
    outer color=red!15,
    draw=red,
    very thick
] (V3) -- (V5) -- (V9) -- (V7) -- (V4) -- cycle;
\draw[very thick] (V1) -- (V2) (V1) -- (V5) (V1) -- (V6) (V1) -- (V9) (V1) -- (V10);
\draw[very thick] (V2) -- (V7) (V2) -- (V8) (V2) -- (V9) (V2) -- (V10);
\draw (V3) -- (V6) (V3) -- (V12); %
\draw (V4) -- (V8) (V4) -- (V12); %
\draw (V5) -- (V6); %
\draw (V6) -- (V10) (V6) -- (V12);
\draw (V7) -- (V8); %
\draw (V8) -- (V10) (V8) -- (V12);
\draw (V10) -- (V12);

\draw[very thick] (V5) -- (V6) (V6) -- (V10) (V8) -- (V10) (V7) -- (V8);

\foreach \i in {1,...,12}{
  \fill (V\i) circle (0.02);
}
\fill (V12) circle (0.02);

\end{tikzpicture}
	\end{minipage}
	\caption{On the left, a standard coning on a pentagon, potentially generating a square with the already existing blue edges. On the right, a representation of the collar-cone on a pentagon as a truncated icosahedron. This does not introduce any square.}
	\label{fig:truncated-icosahedron}
\end{figure}
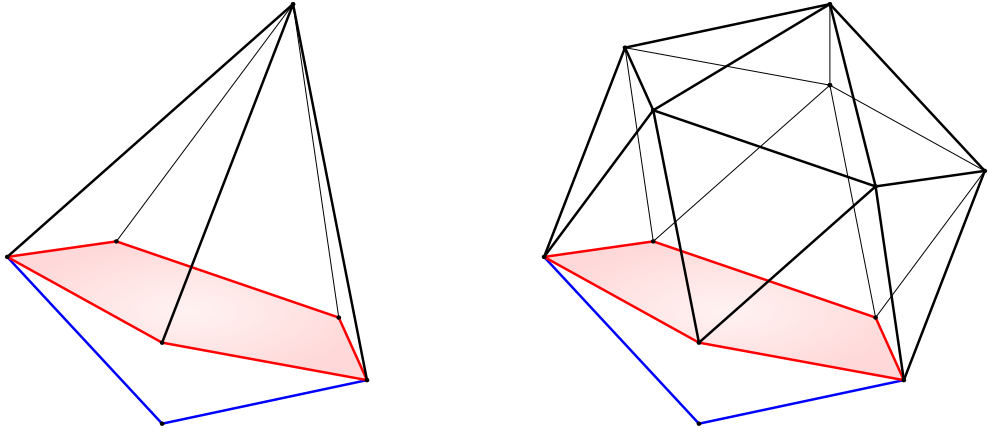

We now show that collar-coning an $n$-cycle preserves the desired properties of $M$. %

\begin{lemma} \label{cone}
	If $M$ is a \fns simplicial complex which defines a Coxeter group of $\vcd=d \geq 3$, then collar-coning any full $n$-cycle $\gamma_n$ for $n \geq 5$ results in a simplicial complex $M \cup C$ which
	\begin{enumerate}
		\item \label{fnscone} is \fns;
		\item \label{vcdcone} defines a Coxeter group of $\vcd = d$;
		\item \label{fullsubcx} contains $M$ as a full subcomplex.

	\end{enumerate}
\end{lemma}

\begin{proof}
	\cref{fnscone}: To see that the new complex is flag, note that no added vertex in $C$ is attached to more than 2 vertices of $M$, and the corresponding triangles have been added. The cone itself is a flag complex.

	To see that the new complex has no squares, note that the centre of the collared cone does not appear in a $4$-cycle since, among its neighbours, any two which have a second neighbouring vertex in common are in fact joined by an edge. The neighbours of the centre similarly have no neighbouring vertex in common apart from the centre unless they are adjacent.

	\cref{vcdcone}: We first show that for any simplex $\sigma \in M \cup C$, the homology of $(M \cup C) \setminus \sigma$ vanishes in dimension $k \geq d$. We distinguish two cases.

	Suppose first that $\sigma \cap C \neq \varnothing$. By construction, $\sigma \cap C$ is a single non-empty simplex $\tau$, and $\tau \cap \partial C$ is either a simplex or empty.

	Recall that by definition $C \setminus \tau = C \setminus N$, where $N$ is the open $1$-neighbourhood of $\tau$.
	Note that $N$ is a $2$-dimensional ball, and $N \cap \partial C$ is either empty or an open segment.
	In both cases, we have that $C \setminus N$ deformation retracts on $\partial C \setminus (N \cap \partial C)$.

	Therefore we have that $(M \cup C) \setminus \sigma$ deformation retracts onto $M \setminus(M \cap \sigma)$. $(M \cap \sigma)$ is a simplex of $M$, so by hypothesis on $M$, the $k$-th homology of $M \setminus(M \cap \sigma)$ vanishes for $k \geq d$.

	Now suppose that $\sigma \cap C = \varnothing$. Then
	\[(M \cup C) \setminus \sigma= (M \setminus \sigma) \cup C\]
	where $C$ is attached by coning the same $n$-cycle, considered as a subset of $M \setminus \sigma$. The Mayer-Vietoris sequence gives
	\[
		\begin{tikzcd}
			\cdots \ar [r] & H_k(\gamma_n) \ar[r] & H_k(C) \oplus H_k(M \setminus \sigma) \ar[r] & H_k((M \setminus \sigma) \cup C) \ar[r] & H_{k-1}(\gamma_n)
		\end{tikzcd}\]
	Since $\gamma_n$ is homeomorphic to a circle and the cone $C$ is contractible, $H_{k-1}(\gamma_n), H_k(\gamma_n),$ and $H_k(C)$ all vanish when $k \geq 3$, so $H_k(M \setminus \sigma) \cong H_k((M \setminus \sigma) \cup C)$ in this range. If, in addition, $k \geq d$, then $H_k(M \setminus \sigma)=0$ by hypothesis, so $H_k((M \setminus \sigma) \cup C)$ vanishes as well, as desired.

	Finally, by hypothesis there exists a simplex $\sigma_0$ of $M$ such that $H_{d-1}(M \setminus \sigma_0) \neq 0$.
	We claim that $H_{d-1}((M \cup C) \setminus \sigma_0)$ is also nontrivial. If $\sigma_0$ intersects $C$, then as above $(M \cup C) \setminus \sigma_0$ deformation retracts onto $M \setminus \sigma_0$. Otherwise, the Mayer-Vietoris sequence yields
	\[
		\begin{tikzcd}
			\cdots \ar [r] & H_{d-1}(\gamma_n) \ar[r] &H_{d-1}(M \setminus \sigma) \ar[r] & H_{d-1}((M \setminus \sigma) \cup C) \ar[r] & H_{d-2}(\gamma_n)
		\end{tikzcd}\]
	As above, $H_{d-1}(\gamma_n)$ must vanish, so $H_{d-1}(M \setminus \sigma)$ injects into $H_{d-1}((M \setminus \sigma) \cup C)$ and we are done.

	\cref{fullsubcx}: A routine check shows that $M$ embeds as a full subcomplex.
\end{proof}

\begin{lemma}\label{full-embedded-generators}
	Let $M$ be a flag-no-square simplicial complex, and let $\gamma_1, \dots, \gamma_k$ be a sequence of full, embedded cycles in $M$. Let $M'$ be the simplicial complex obtained by coning along the $\gamma_i$. Let $\sigma$ be a (possibly empty) simplex in $M$. There exist full embedded cycles $\gamma'_1, \dots, \gamma'_\ell$ in $M \setminus \sigma $ that normally generate $\pi_1(M' \setminus \sigma)$.
\end{lemma}

\begin{remark}
	In the lemma above, the cycles $\gamma_1, \dots, \gamma_k$ are not necessarily distinct. If a cycle appears more than once, then multiple collared cones are attached to it in $M'$.
\end{remark}

\begin{proof}
	Choose generators for $\pi_1(M \setminus \sigma)$, and take loops $S' = \set{\gamma'_1, \dots, \gamma'_\ell}$ of minimal length representing the generators in the $1$-skeleton of $M \setminus \sigma$ of $M$.
	Since $M$ is flag, each cycle must be of length $>3$, and since $M$ has no squares each cycle must be of length $>4$.

	We may assume the loops are embedded, as if some loop $\gamma'_i$ has self-intersections we may divide it into smaller loops, and then replace $\gamma'_i$ with the union of these smaller loops in $S'$ (this corresponds to changing the generating set of $\pi_1(M \setminus \sigma)$. For the same reason, we may assume that the cycles are full: if a cycle $\gamma'_i$ is not full, then there are two vertex which are adjacent in $M$ but not in the cycle, so we may find a shortcut and cut $\gamma'_i$ into smaller cycles. Note that at every step the quantity $\sum_{\gamma' \in S'} 3^{\abs{\gamma'}}$ decreases, so this procedure will eventually terminate.

	Now the conclusion follows as the map induced map $\pi_1(M \setminus  \sigma ) \to \pi_1(M' \setminus \sigma) $ is surjective, and that $M \setminus \sigma$ is a full subcomplex of $M' \setminus \sigma$, so the cycles remain full. Note that some of the cycles may be homotopically trivial in $M' \setminus \sigma$, in which case we may safely remove them.
\end{proof}

We are now in a position to prove \Cref{coning-procedure}.
\begin{proof}[Proof of \Cref{coning-procedure}]
	Given $M$, choose full embedded cycles $\gamma_1, \dots, \gamma_k$ in $M$ generating $\pi_1(M)$ as a normal subgroup, as per \cref{full-embedded-generators} (applied with $\sigma$ being empty and no cones). Let $M'$ be the simplicial complex obtained by coning on each $\gamma_i$ twice.

	Now enumerate the simplices $s_1, \dots, s_\ell$ of $M'$. Define $M'_0=M'$. Now for every $i$, find full embedded cycles in $M \setminus \sigma_i \subseteq M'_{i-1} \setminus \sigma $ that generate $\pi_1(M'_{i-1})$ as a normal subgroup, by \cref{full-embedded-generators}. Now let $M'_i$ be obtained from $M'_{i-1}$ by coning these cycles. Let $N = M'_\ell$.

	We claim that $N$ satisfies the hypotheses.

	Since $N$ is obtained from $M$ by attaching some cones, \cref{cone} ensures that $M$ is full in $N$, $N$ is flag no square, and $M$ and $N$ define Coxeter groups of the same $\vcd$.

	We now check that $N$ has no $1$-disconnecting simplices. By construction, $N$ is simply connected. Let $\sigma$ be a simplex of $N$. If it intersects any added cone $C$ then as above $N \setminus \sigma$ is homotopic to $N \setminus (\interior(C) \cup \sigma)$. %
	Let the cones it intersects be $C_1 \dots C_a$. Then $N \setminus \sigma$ is homotopic to $N \setminus (\bigcup_{i=1}^a C_i \cup \sigma)$. But this latter complex is $M\setminus (M \cap \sigma)$ with some cones glued on. By construction, there were some glued on cones which killed $\pi_1(M\setminus M \cap \sigma)$. These have to be disjoint from $\sigma$, so they remain glued on and hence $N \setminus \sigma$ is simply-connected.
\end{proof}
\begin{remark}
	In the previous procedure, we may also simply cone every full cycle twice, regardless on whether it is homotopically trivial or not, but the given procedure gives finer control for potential future applications.
\end{remark}
\begin{remark}
	It is very easy to produce \fns complexes with no 0-disconnecting simplices as input for \cref{coning-procedure}: given a \fns simplicial complex $M$, enumerate all 0-disconnecting simplices as in the proof of \cref{coning-procedure} and add sufficiently long paths connecting the resulting path components. This preserves the same properties listed in \cref{cone}.
\end{remark}

\section{Proofs of the main results}\label{sec:theorems}
In this section we prove the results stated in the introduction.
\begin{proof}[Proof of \cref{main:fp-fibring-qi}]
	For every $d \geq 3$, \cite{LMSSW} show the existence of \fns simplicial complexes with no 0-disconnecting simplices that define a hyperbolic RACG of $\vcd =d$. Applying \cref{coning-procedure} gives rise to \fns simplicial complexes with no 1-disconnecting simplices of $\vcd =d$. By \cref{verythick}, the corresponding enlargement virtually fibres with finitely presented kernel. This already yields \Cref{main:fp-fibring}.

	For $l \geq 5$, the graph product defined on a cycle of length $l$ with vertex groups $(\mathbb{Z}/2\mathbb{Z})^N$ is a cocompact lattice $G_{N,l}$ in a Bourdon building \cite{C}. See \cite{B} for the definition of a Bourdon building. For fixed $l$, the conformal dimension of the boundary of the building, and hence that of the group $G_{N,l}$, goes to infinity with $N$. We note that although Bourdon originally considers graph products with vertex groups $\mathbb{Z}/2^N\mathbb{Z}$, these are commensurable to the graph products with vertex groups $(\mathbb{Z}/2\mathbb{Z})^N$ \cites{Kim12,LP, JS01}, hence have the same conformal dimension.

	We may assume that the input complex $X$ contains an $l$-cycle for some $l \geq 5$ (c.f. \cite[Remark 1.1]{LMSSW}). By \cref{verythick} we can choose an enlargement $\enlargement[\pi_N] X$ so that the thickened complex defines a word-hyperbolic Coxeter group $G$ that contains $G_{N,l}$ as a special subgroup for all $N$ sufficiently large. Since special subgroups are convex, by monotonicity this gives a lower bound on the conformal dimension of $G$. Letting $N$ tend to infinity gives infinitely many different values for the conformal dimension of $G$, so no two are quasi-isometric.
\end{proof}
\begin{remark}
	The argument about conformal dimension is identical to that in \cite{CDSS}. In their preprint they state that it is unclear how to make the argument about conformal dimension compatible with the thickening procedures of \cite{Osajda} and \cite{LMSSW}, but it appears to be compatible with our enlargement procedure.
\end{remark}

\subsection{Euler characteristic}\label{sec:trash}%
To prove the corollaries, we need a way to compute the Euler characteristic of the groups we produce above. To compute the Euler characteristic of a right-angled Coxeter group, we can apply the following formula. In the following sums, the empty simplex is always included.

\begin{theorem}[\cite{Chiswell}]\label{euler-char-racg}
	Let $\scx$ be a flag simplicial complex, and let $W$ be the \RACG defined by $\scx$. Then
	\[
		\chi(W) = \sum_{\sigma \in \scx} \left(-\frac 12\right)^{\dim \sigma + 1}.
	\]
\end{theorem}

From this we get the following formula for the Euler characteristic for the group defined by the enlargement.

\newcommand{\thickness}[1]{\thicksize(#1)}
\begin{proposition} \label{euler-char-thickening}
	Let $\scx$ be a flag simplicial complex, and consider an enlargement $\enlargement[\alpha]\scx$. For each vertex $v \in \scx$, denote by $\thickness v \in \NN$  the number of vertices above $v$, that is, the cardinality of $\alpha^{-1}(v)$. Then the Coxeter group $W_{\enlargement[\alpha]\scx}$ satisfies
	\[
		\chi(W_{\enlargement[\alpha]\scx}) = \sum_{\sigma \in \scx} \prod_{v \in \sigma} \left(\frac 1{2^{\thickness v}}-1\right).
	\]
\end{proposition}

\begin{proof}
	We rewrite the formula in \cref{euler-char-racg} as
	\[
		\chi(W_{\enlargement[\alpha]\scx}) = \sum_{\tau \in \enlargement\scx} \prod_{v \in \tau} \left(-\frac12\right) = \sum_{\sigma \in \scx} \sum_{\alpha(\tau)=\sigma} \prod_{v \in \tau} \left(-\frac12\right).
	\]
	Now fix a simplex $\sigma \in \scx$, with vertices $v_1, \dots, v_k$. For every $k$-tuple of integers $i_1, \dots, i_k$ with $1 \leq i_j \leq \thickness{v_j}$, there are $\prod_{j=1}^k \binom{\thickness{v_j}}{i_j}$ simplices $\tau \in \enlargement\scx$ with $\alpha(\tau)=\sigma$ and exactly $i_j$ many vertices above $v_j$ for each $j$. So we can write
	\begin{align*}
		\sum_{\alpha(\tau)=\sigma} \prod_{v \in \tau} \left(-\frac12\right) & = \sum_{i_1 = 1}^{\thickness{v_1}} \dots \sum_{i_k = 1}^{\thickness{v_k}} \left(\prod_{j=1}^k \binom{\thickness{v_j}}{i_j} \right)\cdot \left(-\frac12\right)^{i_1 + \dots + i_k} \\
		                                                                    & = \prod_{j=1}^k \left(\sum_{i_j=1}^{\thickness{v_j}}\binom{\thickness{v_j}}{i_j} \cdot  \left(-\frac 12\right)^{i_j}\right)                                                       \\
		                                                                    & = \prod_{j=1}^k \left(\left(1-\frac12\right)^{\thickness {v_j}}-1\right)                                                                                                          \\
		                                                                    & = \prod_{j=1}^k \left(\frac1{2^{\thickness {v_j}}}-1\right).
	\end{align*}
	The result now follows.
\end{proof}

\subsection{Exotic finiteness properties}

In this subsection we prove the results concerning exotic finiteness properties of the kernel.

\begin{proof}[Proof of \cref{main:fp-fibring-notFn}]
	As in the proof of \cref{main:fp-fibring-qi}, we start with a \fns simplicial complex $M$ coming from \cite{LMSSW}, collar-cone cycles of $M$, and finish with the final enlargement. We show that it is possible to obtain a group with non-zero Euler characteristic by simply extending the coning phase of our construction.

	Let $M$ be a \fns simplicial complex. Adding an $m$-simplex to $M$ changes the Euler characteristic of the RACG associated to $\enlargement[\pi_\thicksize] M$ by adding a term that, when $\thicksize$ is large, can be approximated by $(-1)^{m+1}$ by \cref{euler-char-thickening}.
	The procedure of collar-coning an $n$-cycle adds $n+1$ vertices, $4n$ 1-simplices, and $3n$ 2-simplices; therefore, collar-coning one more $n$-cycle of $M$ decreases the Euler characteristic of the constructed group by approximately 1.

	Thus, starting with an arbitrary \fns complex $M$ such that $\enlargement[\pi_\thicksize] M$ defines a hyperbolic \RACG of virtual cohomological dimension $n$ that fibres with kernel of type $\finitetype[2]$, we may cone any loop sufficiently many times so that, for large $\thicksize$, the \RACG $G$ associated to $\enlargement[\pi_\thicksize] M$ has negative Euler characteristic. The group $G$ still virtually fibres with kernel of type $\finitetype[2]$, as the hypothesis of \cref{prop:alpha-thickening-k-fibres} are still satisfied.

	All \RACGs are virtually RFRS, so by \cite{F24} this implies that the $L^2$-Betti numbers $b_0^{(2)}(G)$, $b_1^{(2)}(G)$, and $b_2^{(2)}(G)$ vanish. It is known that the $L^2$-Betti numbers vanish above the cohomological dimension, and the alternating sum of the $L^2$-Betti numbers computes the Euler characteristic $\chi(G)$. Since $\chi(G) \neq 0$, there must be some $3 \leq k \leq d$ such that $b_k^{(2)}(G)$ is positive. Again by \cite{F24} this implies that the kernel of the fibring cannot be of type $FP_k(\mathbb{Q})$, hence not of type $F_n$.
	We obtain infinitely many quasi-isometry classes as in the proof of \cref{main:fp-fibring-qi}.

	Let us now assume that $n$ is even. If the kernel were of type $\finitetype[n-1]$, then the only $L^2$-Betti number that could be non-vanishing would be $b_n^{(2)}$, which contributes with positive sign to the Euler characteristic since $n$ is even. This contradicts $\chi(G)$ being (strictly) negative.
\end{proof}
\section{Further directions} \label{sec:next}

Even though it is likely that almost all our groups already fail to satisfy property $F_3$, it might be possible to generalise our collar-coning technique to ensure the absence of $2$-disconnecting simplices, therefore obtaining a different family of groups fibring with kernels having property $F_3$.

Concretely, we may view the collar-cone as a truncation of a flag-no-square triangulation of $S^2$ by removing one of the vertices (see \cref{rmk:Lobell} and \cref{fig:truncated-icosahedron}).
One might attempt to cone generators of $H_2$ with a truncation of a flag-no-square triangulation of $S^3$ to ensure there are no 2-disconnecting simplices. This construction would involve a much higher combinatorial complexity, as this would require us to cone triangulations of $S^2$, which are more complicated than triangulations of $S^1$. While in principle it is possible to envision pushing a similar strategy to even higher finiteness properties, the fact that there are no flag-no-square triangulations of $S^n$ for $n\geq 4$ \cite{PS09} implies that this would likely require using objects that are not topologically a disc to perform the filling. This added topological complexity, on top of the combinatorial complexity, leads the authors to believe this procedure is out of reach with the current technology. %
\medskip

We have shown above that in fact there are many fibrations of hyperbolic groups where the kernel is finitely presented. However, as mentioned in the introduction, the ultimate goal is to produce fibrations with kernel of type $\finitetype$. One necessary condition to get type $\finitetype$ kernel is that the group must have vanishing Euler characteristic. It would be interesting to construct examples where the Euler characteristic is zero, especially in view of the difficulties with generalising the coning construction to higher dimensions as discussed above; however it seems unlikely that the groups constructed in \cite{LMSSW} have zero Euler characteristic, and the coning procedure does not give fine enough control on the Euler characteristic to construct such an example.

\bibliographystyle{alpha}
\bibliography{references}
\end{document}